\documentclass[a4paper,12pt]{amsart}
\usepackage{amssymb}
\usepackage{ifthen}
\usepackage{graphicx}
\nonstopmode \numberwithin{equation}{section}
\newtheorem{thm}[equation]{Theorem}
\newtheorem{cor}[equation]{Corollary}
\newtheorem{lem}[equation]{Lemma}
\newtheorem{prop}[equation]{Proposition}

\newtheorem{conj}{Conjecture}

\theoremstyle{definition}

\newtheorem{prob}[equation]{Problem}

\newenvironment{rem}{%
\bigskip
\noindent \textsl{{\sl Remark. }}}{\bigskip}
\newenvironment{rems}{%
\bigskip
\noindent \textsl{{\sl Remarks. }}}{\bigskip}

\newcounter {own}
\def\theown {\thesection       .\arabic{own}}

\newenvironment{pf}[1][]{%
 \vskip 3mm
 \noindent
 \ifthenelse{\equal{#1}{}}%
  {{\slshape Proof. }}%
  {{\slshape #1.} }%
 }%
{\qed
\medskip
}

\newcounter{alphabet}
\newcounter{tmp}

\newcounter{minutes}
\divide\time by 60
\newcounter{hours}
\multiply\time by 60
\addtocounter{minutes}{-\time}

\def\be{\begin{equation}}
\def\ee{\end{equation}}

\newcommand{\bee}{\begin{enumerate}}
\newcommand{\eee}{\end{enumerate}}

\newcommand{\blem}{\begin{lem}}
\newcommand{\elem}{\end{lem}}
\newcommand{\bthm}{\begin{thm}}
\newcommand{\ethm}{\end{thm}}
\newcommand{\bcor}{\begin{cor}}
\newcommand{\ecor}{\end{cor}}
\newcommand{\beg}{\begin{examp}}
\newcommand{\eeg}{\end{examp}}
\newcommand{\begs}{\begin{examples}}
\newcommand{\eegs}{\end{examples}}
\newcommand{\bdefe}{\begin{defin}}
\newcommand{\edefe}{\end{defin}}
\newcommand{\bprob}{\begin{prob}}
\newcommand{\eprob}{\end{prob}}
\newcommand{\bei}{\begin{itemize}}
\newcommand{\eei}{\end{itemize}}

\newcommand{\bcon}{\begin{conj}}
\newcommand{\econ}{\end{conj}}
\newcommand{\bcons}{\begin{conjs}}
\newcommand{\econs}{\end{conjs}}
\newcommand{\bprop}{\begin{prop}}
\newcommand{\eprop}{\end{prop}}
\newcommand{\br}{\begin{rem}}
\newcommand{\er}{\end{rem}}
\newcommand{\brs}{\begin{rems}}
\newcommand{\ers}{\end{rems}}
\newcommand{\bo}{\begin{obser}}
\newcommand{\eo}{\end{obser}}
\newcommand{\bos}{\begin{obsers}}
\newcommand{\eos}{\end{obsers}}
\newcommand{\bpf}{\begin{pf}}
\newcommand{\epf}{\end{pf}}
\newcommand{\ba}{\begin{array}}
\newcommand{\ea}{\end{array}}
\newcommand{\beq}{\begin{eqnarray}}
\newcommand{\beqq}{\begin{eqnarray*}}
\newcommand{\eeq}{\end{eqnarray}}
\newcommand{\eeqq}{\end{eqnarray*}}

\begin{document}

\title{Logarithmic coefficients for certain subclasses of close-to-convex functions}

\author{U. Pranav Kumar}
\address{Pranav Kumar Upadrashta,
Department of Mechanical Engineering,
Indian Institute of Technology Kharagpur,
Kharagpur-721 302, West Bengal, India.}
\email{upk1993@gmail.com}

\author{A. Vasudevarao}
\address{A. Vasudevarao,
Department of Mathematics,
Indian Institute of Technology Kharagpur,
Kharagpur-721 302, West Bengal, India.}
\email{alluvasu@maths.iitkgp.ernet.in}


\def\thefootnote{}
\footnotetext{ {\tiny File:~\jobname.tex,
printed: \number\year-\number\month-\number\day,
          \thehours.\ifnum\theminutes<10{0}\fi\theminutes }
} \makeatletter\def\thefootnote{\@arabic\c@footnote}\makeatother

\thanks{}

\maketitle
\pagestyle{myheadings}
\markboth{U. Pranav Kumar and A. Vasudevarao }{Logarithmic coefficients}

\begin{abstract}
Let $\mathcal{S}$ denote the class of functions analytic and univalent (i.e. one-to-one) in the unit disk 
$\mathbb{D}=\{z\in\mathbb{C}:\, |z|<1\}$ normalized by $f(0)=0=f'(0)-1$. The logarithmic coefficients 
$\gamma_n$ of $f\in\mathcal{S}$ are defined  by
$\log \frac{f(z)}{z}= 2\sum_{n=1}^{\infty} \gamma_n z^n.$
In the present paper, we  determine the sharp upper bounds for
$|\gamma_1|$,  $|\gamma_2|$ and $|\gamma_3|$  when $f$ belongs to some
familiar subclasses of close-to-convex  functions.
\end{abstract}

\section{Introduction}\label{Introduction}

Let $\mathbb{D}:=\{z\in\mathbb{C}:\, |z|<1\}$ denote the unit disk in the complex plane $\mathbb{C}$.  A single-valued function $f$ is said to be univalent in a domain $\Omega\subseteq\mathbb{C}$ if it never takes the same value twice, that is, if $f(z_1)=f(z_2)$ for $z_1,z_2 \in\Omega$ then $z_1=z_2$.  Let   $\mathcal{A}$ denote the class of analytic functions $f$ in $\mathbb{D}$ normalized by $f(0)=0=f'(0)-1$. If $f\in\mathcal{A}$ then $f(z)$ has the following representation
\begin{equation}\label{i01}
f(z)= z+\sum_{n=2}^{\infty}a_n z^n.
\end{equation}
Let $\mathcal{S}$ denote  the class of univalent  functions in $\mathcal{A}$.  A domain $\Omega\subseteq\mathbb{C}$ is said to be a starlike domain with respect to a point $z_0\in\Omega$
if the line segment joining $z_0$ to any point in $\Omega$  lies in $\Omega$. If $z_0$ is the origin then we say that $\Omega$ is a starlike domain.
A function $f\in\mathcal{A}$ is said to be a starlike function
if  $f(\mathbb{D})$  is a starlike domain. We denote by $\mathcal{S}^*$ the class of starlike functions $f$ in $\mathcal{S}$. It is well-known that \cite{Duren-book} a function $f\in\mathcal{A}$ is in $\mathcal{S}^*$ if and only if
$$
{\rm Re\,}\left(\frac{zf'(z)}{f(z)}\right)>0 \quad\mbox{ for } z\in\mathbb{D}.
$$
A domain $\Omega$ is said to be convex if it is starlike with respect to each point of $\Omega$.
A function $f\in\mathcal{A}$ is said to be convex if $f(\mathbb{D})$ is a convex domain. We denote the class of convex univalent functions in $\mathbb{D}$ by $\mathcal{C}$. A function $f\in\mathcal{A}$ is in $\mathcal{C}$ if and only if
$$
{\rm Re\,}\left(1+\frac{z f''(z)}{f'(z)}\right)>0 \quad\mbox { for } z\in\mathbb{D}.
$$
It is well-known that  $f\in\mathcal{C}$ if and only if $zf'\in\mathcal{S}^*$.\\


A function $f\in\mathcal{A}$ is said to be close-to-convex (having argument $\alpha\in(-\pi/2,\pi/2)$) with respect to $g\in\mathcal{S}^*$  if
$$
{\rm Re\,}\left(e^{i\alpha}\frac{zf'(z)}{g(z)}\right)>0 \quad\mbox { for } z\in\mathbb{D}.
$$
We denote the class of all such functions by $\mathcal{K}_{\alpha}(g)$. Let
$$
\mathcal{K}(g):= \bigcup_{\alpha\in(-\pi/2,\,\pi/2)} \mathcal{K}_{\alpha}(g) \quad\mbox{ and }\quad \mathcal{K}_{\alpha}:= \bigcup_{g\in\mathcal{S}^*} \mathcal{K}_{\alpha}(g)
$$
be the classes of close-to-convex functions with respect to $g$ and close-to-convex functions with
argument $\alpha$, respectively. Let
$$
\mathcal{K}:= \bigcup_{\alpha\in(-\pi/2,\,\pi/2)} \mathcal{K}_{\alpha}= \bigcup_{g\in\mathcal{S}^*} \mathcal{K}(g)
$$
denote  the class of  close-to-convex functions in $\mathcal{A}$. It is well-known that every close-to-convex function is univalent in $\mathbb{D}$ \cite{Kaplan-52}.
A domain $\Omega\subseteq\mathbb{C}$ is said to be linearly accessible if its complement is the union of a family of non-intersecting half-lines.
A function $f\in\mathcal{S}$ whose range is linearly accessible is called a linearly accessible function.
Kaplan's theorem \cite{Kaplan-52} makes it seem plausible that the class of linearly accessible family and the class $\mathcal{K}$ coincide.
In fact, Lewandowski \cite{Lewandowski-1958}  has observed that the class $\mathcal{K}$
is the same as the class of linearly accessible functions introduced by Biernacki \cite{Biernacki-1937} in 1936. In 1962,
Bielecki and Lewandowski \cite{Bi-lewan} proved that every function in the class $\mathcal{K}$ is  linearly accessible.

Let $\mathcal{P}$ denote the class of analytic functions $h(z)$ of the form
\begin{equation}\label{positiveRealPart}
h(z)=1+\sum_{n=1}^\infty c_n z^n
\end{equation}
such that ${\rm Re}\, h(z)>0$ in $\mathbb{D}$.
To prove our main results we need the following results.
\begin{lem}\cite{Li-Zl-1984}\label{thm-4}
Let  $h \in \mathcal{P}$ be of the form (\ref{positiveRealPart}). Then
\begin{eqnarray*}\label{eq12}
2c_2&=&c_{1}^{2}+x(4-c_{1}^{2})\\
4c_3&=&c_{1}^{3}+2(4-c_{1}^{2})c_1x-c_1(4-c_{1}^{2})x^2+2(4-c_{1}^{2})(1-|x|^2)t.
\end{eqnarray*}
for some complex valued $x$ and  $t$ with $|x|\leq 1$ and   $|t|\leq 1$.
\end{lem}

\begin{lem}\cite[pp 166]{pommerenke}
\label{lem2}
Let  $h \in \mathcal{P}$ be of the form (\ref{positiveRealPart}). Then
\begin{equation*}
\Bigl|c_2-{\tiny \frac{c_1^2}{2}}\Bigr| \leq 2 - {\tiny \frac{|c_1|^2}{2}}.
\end{equation*}
\end{lem}
The inequality is sharp for functions $L_{t,\theta}(z)$ of the form
\begin{equation*}
L_{t,\theta}(z)= t\left(\frac{1+e^{i\theta}z}{1-e^{i\theta}z}\right)+(1-t)\left(\frac{1+e^{i2\theta}z^2}{1-e^{i2\theta}z^2}\right).
\end{equation*}

\begin{lem}\cite{Ma-Minda-1992}\label{lem3}
Let $h\in\mathcal{P}$ be of the form (\ref{positiveRealPart}) and $\mu$ be a complex number. Then
\begin{equation*}
|c_2-\mu c_1^2|\le 2\, \max\{1,|2\mu-1|\}.
\end{equation*}
The result is sharp for the functions given by $p(z)=\frac{1+z^2}{1-z^2}$ and $p(z)=\frac{1+z}{1-z}$.
\end{lem}

Given a function $f \in \mathcal{S}$, the coefficients $\gamma_n$  defined by
\begin{equation}\label{i05}
\log \frac{f(z)}{z} = 2\sum_{n=1}^{\infty} \gamma_n z^n
\end{equation}
are called the logarithmic coefficients of $f(z)$. The logarithmic coefficients are central to the theory of univalent functions for their role in the proof of Bieberbach conjecture.
Milin conjectured that for $f\in\mathcal{S}$ and $n\ge 2$,
$$
\sum_{m=1}^{n}\sum_{k=1}^{m} \left(k|\gamma_k|^2-\frac{1}{k}\right)\le 0.
$$
Since  Milin's conjecture implies Bieberbach conjecture,
in 1985, De Branges proved  Milin conjecture to give an affirmative proof of the Bieberbach conjecture \cite{Branges-1985}.

By differentiating (\ref{i05}) and equating coefficients we obtain
\begin{eqnarray}
\gamma_1&=&\frac{1}{2} a_2 \label{i10}\\
\gamma_2&=&\frac{1}{2}(a_3-\frac{1}{2}a_2^2) \label{i15}\\
\gamma_3&=&\frac{1}{2}(a_4-a_2a_3+\frac{1}{3}a_2^3). \label{i20}
\end{eqnarray}
It is evident from (\ref{i10})  that $|\gamma_1|\le 1$  if  $f\in\mathcal{S}$. An application of Fekete-Szeg\"{o} inequality \cite[Theorem 3.8]{Duren-book} in (\ref{i15}) yields the following sharp estimate
$$
|\gamma_2|\le \frac{1}{2}(1+2e^{-2})=0.635\ldots\quad\mbox { for } f\in\mathcal{S}.
$$
The problem of finding the sharp upper bound  for $|\gamma_n|$ for  $f\in\mathcal{S}$ is still open for  $n\ge 3$.
The sharp upper bounds for  modulus of logarithmic coefficients are known for functions in very few subclasses of $\mathcal{S}$.
For the Koebe function $k(z)=z/(1-z)^2$, the logarithmic coefficients are $\gamma_n=1/n$. Since the Koebe function $k(z)$ plays the role of extremal function for most of the extremal problems in the class $\mathcal{S}$, it is expected that $|\gamma_n|\le \frac{1}{n}$ holds for functions in the class $\mathcal{S}$. However, this is not true in general. Indeed, there exists a bounded function $f$ in the class $\mathcal{S}$ with logarithmic coefficients $\gamma_n\ne O(n^{-0.83})$ (see \cite[Theorem 8.4]{Duren-book}).
A simple exercise shows that $|\gamma_n| \leq 1/n$ for functions in $\mathcal{S}^*$ and the equality holds for the Koebe function. Consequently, attempts have been made to find bounds for logarithmic coefficients for close-to-convex functions  in the unit disk $\mathbb{D}$.
Elhosh \cite{Elhosh-1996} attempted to extend the result $|\gamma_n|\leq 1/n$ to the class  $\mathcal{K}$. However
Girela \cite{Girela-2000} pointed out an error in the proof and proved that for every $n \geq 2$ there exists a function $f$ in $\mathcal{K}$ such that $|\gamma_n|\geq 1/n$. Ye \cite{Ye-2008} provided an estimate for $|\gamma_n|$ for functions $f$ in the class $\mathcal{K}$, showing that $|\gamma_n| \leq An^{-1} \log n$ where $A$ is a constant. The sharp inequalities are known for sums involving logarithmic coefficients (see \cite{Duren-book,Duren-Leung-1979}).
For $f \in \mathcal{S}$, Roth \cite{Roth-2007} proved the   following sharp inequality
$$ \sum_{n=1}^{\infty} \left(\frac{n}{n+1}\right)^2|\gamma_n|^2 \leq 4 \sum_{n=1}^{\infty} \left(\frac{n}{n+1}\right)^2 \frac{1}{n^2}=\frac{2\pi ^2-12}{3}.
$$

Recently, it has been proved that   $|\gamma_3| \leq 7/12$  for functions in the class $\mathcal{K}_0$  with the additional assumption that the second coefficient of the corresponding starlike function $g(z)$ is real \cite{Thomas-2016}. However this bound is not sharp.
Enough emphasis cannot be laid  on this fact as it highlights nature of complexity involved in obtaining the sharp upper bound  for $|\gamma_3|$.
More recently Firoz and Vasudevarao \cite{Firoz-Vasu-2016} improved the bound on $|\gamma_3|$ by proving $|\gamma_3| \leq \frac{1}{18}(3+4\sqrt{2})=0.4809$ for functions $f$ in the class $\mathcal{K}_0$ without the assumption requiring the second coefficient of the corresponding starlike function $g(z)$ be real.
However,  this improved  bound is still not sharp. Consequently, the problem of finding the sharp upper  bound for $|\gamma_3|$ for the classes $\mathcal{K}_0$ as well as $\mathcal{K}$ is still open.

In the present paper we consider the following  three familiar subclasses of close-to-convex  functions
\begin{eqnarray*}
\mathcal{F}_1:&=&\left\{f\in\mathcal{A}:\textrm{Re\,}(1-z)f'(z)>0 \quad\mbox{ for } z\in\mathbb{D}\right\}\\
\mathcal{F}_2:&=&\left\{f\in\mathcal{A}:\textrm{Re\,}(1-z^2)f'(z)>0 \quad\mbox{ for } z\in\mathbb{D}\right\}\\
\mathcal{F}_3:&=&\left\{f\in\mathcal{A}:\textrm{Re}\,(1-z+z^2)f'(z)>0 \quad\mbox{ for } z\in\mathbb{D} \right\}.
\end{eqnarray*}
The region of variability for the classes $\mathcal{F}_1,\mathcal{F}_2$ and $\mathcal{F}_3$ have been extensively studied by Ponnusamy,  Vasudevarao and  Yanagihara (\cite{Ponnusamy-Vasu-Yanagihara-2008}, \cite{Ponnusamy-Vasu-Yanagihara-2009}).
The main aim of this paper is to determine the sharp upper bounds for $|\gamma_1|$,  $|\gamma_2|$ and  $|\gamma_3|$ for  functions $f$ in the classes
$\mathcal{F}_1,\mathcal{F}_2$ and $\mathcal{F}_3$.

\section{Main Results}
Throughout the remainder of this paper, we assume that $f\in \mathcal{K}_0$ and $h \in \mathcal{P}$ have
the series representations (\ref{i01}) and  (\ref{positiveRealPart}) respectively. Further, assume that $g \in \mathcal{S}^*$ has
the following series representation:
\begin{equation}\label{eq6}
g(z) = z+\sum_{n=2}^{\infty} b_nz^n.
\end{equation}
It is not difficult to see that the function $H_{t,\mu}(z)$ given by
\begin{eqnarray*}
H_{t,\mu}(z) &=& (1-2t)\left(\frac{1+z}{1-z}\right)+t\left(\frac{1+\mu z}{1-\mu z}\right)+t\left(\frac{1 +\overline{\mu} z}{1-\overline{\mu}z}\right)
\end{eqnarray*}
belongs to the class $\mathcal{P}$ for $0\leq t \leq 1/2$ and $|\mu|=1$.
Since $f\in \mathcal{K}_0$, there exists an $h \in \mathcal{P}$ such that
\begin{equation}\label{eq7}
z f'(z) = g(z)h(z).
\end{equation}
Using the representations (\ref{i01}), (\ref{positiveRealPart}) and (\ref{eq6})  in (\ref{eq7}) we obtain
\begin{equation}\label{eq8}
z+\sum_{n=2}^{\infty}n a_nz^n= \left(z+\sum_{n=2}^{\infty}b_nz^n\right)\left(1+\sum_{n=1}^{\infty}c_nz^n\right).
\end{equation}
Comparing the coefficients on both the sides of (\ref{eq8}), we obtain
\begin{eqnarray}
2a_2 &=& b_2 + c_1\label{eq9}\\
3a_3 &= & b_3+b_2c_1+c_2\label{eq10}\\
4a_4&=& b_4+c_1b_3+c_2b_2+c_3\label{eq11}.
\end{eqnarray}
A substitution of   (\ref{eq9}) in (\ref{i10}) gives
\begin{equation}\label{gamma1}
\gamma_1=\frac{1}{4}\left(b_2+c_1\right).
\end{equation}
An application of the triangle inequality to (\ref{gamma1}) gives
\begin{equation}\label{gamma11}
4|\gamma_1| \leq |b_2|+|c_1|.
\end{equation}
Substituting   (\ref{eq9}) and (\ref{eq10})  in (\ref{i15}), we obtain
\begin{equation}\label{gamma2}
\gamma_2= \frac{1}{48}\left(8b_3+2b_2c_1+8c_2-3b^2_2-3c_1^2\right).
\end{equation}
Let $c_1=de^{i\alpha}$ and $q=\cos \alpha$ with $0\leq d\leq 2$ and $0 \leq \alpha < 2\pi$. Applying the triangle inequality in conjunction with Lemma \ref{lem2} allows us to rewrite (\ref{gamma2}) as
\begin{equation}\label{modgamma2}
6|\gamma_2| \leq 2- {\tiny \frac{d^2}{2}}+{\tiny \frac{1}{8}}\Bigl|\left(dq+b_2+id\sqrt{1-q^2}\right)^2+(8b_3-4b_2^2)\Bigr|.
\end{equation}
Substituting   (\ref{eq9}), (\ref{eq10}) and (\ref{eq11}) in (\ref{i20}),  we obtain
\begin{equation}\label{eq14}
\gamma_3=\frac{1}{48}\left(6c_3-b_2^2c_1-b_2c_1^2+2b_2c_2+2b_3c_1+b_2^3-4b_3b_2+6b_4+c_1^3-4c_1 c_2\right).
\end{equation}
A simple application of Lemma \ref{thm-4} to (\ref{eq14}) shows that
\begin{align}\label{gamma3}
96\gamma_3 &=6t(1-|x|^2)(4-c_1^2)+c_1^3+(4b_3-2b_1^2)c_1+(2b_2^3-8b_2b_3+2b_4) \\
           &\qquad \qquad \qquad +x(4-c_1^2)(2b_2+2c_1-3c_1x). \nonumber
\end{align}
Let $b_n$ be real for all $n \in \mathbb{N}$. Let $c_1=c$ and  assume that  $0 \leq c \leq 2$. Let $x=re^{i\theta}$ and $p=\cos \theta$ with $0\leq r\leq 1$ and $0\leq \theta< 2\pi$.
Taking modulus on both the sides of (\ref{gamma3}) and applying the triangle inequality we obtain
\begin{equation}\label{modgamma3}
 96|\gamma_3| \leq 6(1-r^2)(4-c^2)+|\phi(c,r,p)|
\end{equation}
where $$\phi(c,r,p)= c^3+(4b_3-2b_1^2)c+(2b_2^3-8b_2b_3+2b_4)+re^{i\theta}(4-c^2)(2b_2+2c-3cre^{i\theta}).$$
\begin{thm}\label{thm-5}
Let $f \in \mathcal{F}_1$ be given by (\ref{i01}). Then
\begin{enumerate}
\item[{(i)}] $|\gamma_1|\leq \frac{3}{4}$,
\item[{(ii)}] $|\gamma_2|\leq \frac{4}{9}$.
\item[{(iii)}] If  $1/2\leq a_2\leq 3/2$ then $|\gamma_3| \leq \frac{1}{288} \left(11+15 \sqrt{30}\right)$.
\end{enumerate}
The inequalities are sharp.
\end{thm}

\begin{proof}
Let $f \in \mathcal{F}_1$. Then $f$ is a close-to-convex function with respect to the starlike function $g(z)=z/(1-z)$. In view of (\ref{eq7}) the function $f(z)$ can be written as
\begin{equation}\label{neweq1}
zf'(z)= \frac{z}{1-z} ~h(z).
\end{equation}
As $|c_1| \leq 2$ for $h \in \mathcal{P}$ (see \cite[Ch 7, Theorem 3]{goodmanv1}) a comparison of the R.H.S. of (\ref{eq7}) and (\ref{neweq1}), shows that (\ref{gamma11}) reduces to
\begin{equation}\label{gamma1prob1}
4|\gamma_1| \leq 1+|c_1|\leq 3.
\end{equation}
A function $p \in \mathcal{P}$ having $|c_1| =2$ is given by $p(z)=L_{1,\theta}(z)$ for $0 \leq \theta <2\pi$ and substituting $p(z)$ in place of $h(z)$ in (\ref{neweq1}) determines a function $f\in\mathcal{F}_1$
for which the upper bound on $|\gamma_1|$ is sharp.

In view of (\ref{eq7}) and (\ref{neweq1}),  we can rewrite (\ref{modgamma2}) as
\begin{equation}\label{gamma2prob1}
6|\gamma_2| \leq 2-\frac{|c_1|^2}{2}+\frac{1}{8}\sqrt{(d^2+5+2d q)^2 -16d^2(1-q^2)}=:g(d,q).
\end{equation}
In view of (\ref{gamma2prob1}) it suffices to find points in the square $S:=[0,2]\times [-1,1]$ where $g(d,q)$ attains the maximum value to determine the maximum value of $|\gamma_2|$.
Solving $\frac{\partial g(d,q)}{\partial d}=0$ and $\frac{\partial g(d,q)}{\partial q}=0$ shows that  there is no real valued solution
to the pair of equations. Thus $g(d,q)$ does not attain  maximum in the interior of $S$.

On the side $d=0$, $g(d,q)$ reduces to $g(0,q)=21/8$. On the side $d=2$, $g(d,q)$ can be written as $g(2,q)=\frac{1}{8}\sqrt{80t^2+72t+17}$. An elementary calculation shows that  $\underset{-1\leq q\leq 1}{\mathrm{max}}\;g(2,q)=g(2,1)=1.625$.

On the side $q=-1$, $g(d,q)$ maybe simplified to $g(d,-1)=(21-2d-3d^2)/8$. It is not difficult to see that $g(d,1)$ is decreasing for $c \in [0,2]$. Thus $\underset{0\leq d\leq 2}{\mathrm{max}}\;g(d,-1)=d(0,-1)=21/8=2.625$.

On the side $q=1$, $g(d,q)$ becomes $g(d,1)=(21+2d-3d^2)/8$. An elementary computation shows that $\underset{0\leq d\leq 2}{\mathrm{max}}\;g(d,1)=d(1/3,1)=8/3.$

Thus the maximum value of $g(d,q)$ and consequently that of $|\gamma_2|$ is attained at $(d,q)=(1/3,1)$, i.e., at  $c_1=1/3$. Thus, from (\ref{gamma2prob1}) we obtain $|\gamma_2| \leq 4/9$. Therefore in view of (\ref{neweq1}) and Lemma \ref{lem2} the equality holds in $\rm{(ii)}$ for  the function $\widetilde{F_1}\in\mathcal{F}_1$ such that $z\widetilde{{F_1}'}(z)=z(1-z)^{-1}L_{t,\theta}(z)$ with $t=1/6$ and $\theta=0$.

In view of (\ref{neweq1}), we may rewrite (\ref{modgamma3}) as
\begin{equation}\label{no17}
48|\gamma_3| \leq 3 \left(4-c^2\right) \left(1-r^2\right) + \sqrt{\phi_1(c,r,p)},
\end{equation}
where
\begin{align*}\label{no18}
\phi_1(c,r,p) &= \left(\frac{c^3}{2}+c+3\right)^2+\left(4-c^2\right)^2 r^2 \left(-3 c^2 p r+\frac{9}{4}c^2 r^2+c^2-3 c p r+2 c+1\right)\\
 &\quad +2 \left(\frac{c^3}{2}+c+3\right) \left(4-c^2\right) r \left(\frac{3}{2}c r-3 c p^2r-1+c p+p\right).
\end{align*}
Let $G(c,r,p)= 3\left(4-c^2\right) \left(1-r^2\right) + \sqrt{\phi_1(c,r,p)}$.
Thus it suffices to find points in the closed cuboid $R:=[0,2]\times[0,1]\times[-1,1]$ where $G(c,r,p)$ attains the maximum value.
We  accomplish this  by finding the maximum values in the interior of the six faces,
on the twelve edges and in the interior of $R$.
%

On the face c=0, it can be seen that $G(c,r,p)$ reduces to
\begin{equation}\label{no19}
G(0,r,p)=\sqrt{24 p r+16 r^2+9}+12 \left(1-r^2\right).
\end{equation}
To determine the points on this face where the maxima occur, we solve $\frac{\partial G(0,r,p)}{\partial r}=0$ and $\frac{\partial G(0,r,p)}{\partial p}=0$. The only solution for this pair of equations is $(r,p)=(0,0)$. Thus, no maxima occur in the interior of the face $c=0$.

On the face c=2, $G(c,r,p)$ becomes $G(2,r,p)=9$ and hence
\begin{equation*}
\underset{0<r<1,~ -1< p< 1}{\mathrm{max}} \; G(2,r,p) = 9.
\end{equation*}
On the face $r=0$, $G(c,r,p)$ reduces to
\begin{equation}\label{no20}
G(c,0,p)=12-3 c^2+\frac{1}{2}\left(c^3+2 c+6\right).
\end{equation}
To determine points where maxima occur, it suffices to find points where $\frac{\partial G(c,0,p)}{\partial c}=0$ because  $G(c,0,p)$ is independent of $p$. The set of all such points is $ \{\frac{1}{3}\left(6-\sqrt{30}\right)\}\times \{0\}\times [-1,1]$ and hence $G\left(\frac{1}{3} \left(6-\sqrt{30}\right),0,p\right)=\frac{10 \sqrt{10}}{3 \sqrt{3}}+9=15.0858$.
Thus
\begin{equation*}
\underset{0< c < 2,~ -1< p< 1} {\mathrm{max}} \; G(c,0,p) = \frac{10 \sqrt{10}}{3 \sqrt{3}}+9=15.0858.
\end{equation*}
On the face $r=1$, $G(c,r,p)$ reduces to
\begin{equation}\label{no21}
G(c,1,p)=\sqrt{\psi_1(c,p)+\frac{1}{2} \left(c^2-4\right) \left(c^3+2 c+6\right) \left(6 c p^2-2cp-2p-3 c\right)}
\end{equation}
where
\begin{equation*}
\psi_1(c,p) = \left(\frac{c^3}{2}+c+3\right)^2+\left(c^2-4\right)^2 \left(\frac{1}{4}(c^2-12pc+8c)+1\right).
\end{equation*}
A computation shows that $\frac{\partial G(c,1,p)}{\partial p}=0$ yields
\begin{equation}\label{pindGdp}
p=\frac{2 c^4+2 c^3-5 c^2-2 c+3}{3 c \left(c^3+2 c+6\right)}.
\end{equation}
A more involved computation shows that $\frac{\partial G(c,1,p)}{\partial c}=0$ implies
\begin{alignat}{3}
&(9c^5-12c^3+27c^2-24c-36)p^2-(12c^5+10c^4-52c^3-30c^2+46c+8)p \label{dGdc} \\
&+(6 c^5+5 c^4-42 c^3-33 c^2+57 c+37)=0. &&\nonumber.
\end{alignat}
Substituting (\ref{pindGdp}) in (\ref{dGdc}) and performing a lengthy computation gives
\begin{equation}\label{dGdc2}
\frac{(c^3-7c-3)\zeta_1 (c)}{3 c^2 \left(c^3+2 c+6\right)^2}=0
\end{equation}
where
$$ \zeta_1 (c)=6 c^{10}-5 c^9+20 c^8+86 c^7-49 c^6+257 c^5+623 c^4-629 c^3-1095 c^2-60 c+36.$$
The numerical solutions of (\ref{dGdc2}) such that $0< c <2$ are $c\approx 0.151355$ and $c \approx 1.30718$. Substituting these values of $c$ in (\ref{pindGdp}) gives $p \approx 0.904769$ and $p \approx 0.050509$.
The corresponding values of $G(c,1,p)$ are $G(0.151355,1,0.904769)=6.83676$ and $G(1.30718,1,0.050509)=11.2488$ respectively.

 As $G(c,1,p)$ is uniformly continuous on $[0,2]\times \{1\}\times [-1,1]$, the difference between extremum values of $G(c,1,p)$ and either of $6.83676$ or $11.2488$ can be made smaller than an $\epsilon \ll 1$.
%
%
Therefore
\begin{equation}
\underset{0< c< 2,~ -1< p< 1}{\mathrm{max}} \; G(c,1,p) \approx 11.2488.
\end{equation}

On the face $p=-1$, $G(c,r,p)$ reduces to
\begin{equation*}
G(c,r,-1)=\frac{1}{2}(3r^2+2r+1)c^3+(3r^2+r-3)c^2-(6r^2+4r-1)c-(12r^2+4r-15).
\end{equation*}
Now we show that $\frac{\partial G(c,r,-1)}{\partial c}=0$ and $\frac{\partial G(c,r,-1)}{\partial r}=0$ have no solution in the interior of this face. On the contrary, assume that $\frac{\partial G(c,r,-1)}{\partial c}=0$ and $\frac{\partial G(c,r,-1)}{\partial r}=0$ have a solution in the interior of the face $p=-1$. Then $\frac{\partial G(c,r,-1)}{\partial r}=0$ gives
\begin{equation}\label{newereq2}
r = \frac{c+1}{3 (2-c)}.
\end{equation}
By substituting (\ref{newereq2}) in $\frac{\partial G(c,r,-1)}{\partial c}=0$, we obtain  $c=\frac{1}{6} \left(-4\pm\sqrt{190}\right)$, both of which lie outside the range of $c \in [0,2]$.

On the face $p=1$, $G(c,r,p)$ reduces to
\begin{equation*}
G(c,r,1)=\frac{1}{2}(3r^2-2r+1)c^3+(3r^2-r-3)c^2-(6r^2+4r-1)c-(12r^2-4r-15).
\end{equation*}
At the points where $G(c,r,1)$ attains the maximum value, $\frac{\partial G(c,r,1)}{\partial c}$ and $\frac{\partial G(c,r,1)}{\partial r}$ necessarily vanish. The solution to the pair of equations $\frac{\partial G(c,r,1)}{\partial c}=0$ and $\frac{\partial G(c,r,1)}{\partial r}=0$ is $(c,r)=\left(\frac{1}{2}(60-\sqrt{30}),\frac{1}{105}(25-\sqrt{30})\right)$ and subsequently
\begin{equation*}
G\left(\frac{1}{2} \left(6-\sqrt{30}\right),\frac{1}{105} \left(25-\sqrt{30}\right),1\right)=5 \sqrt{\frac{15}{2}}+\frac{11}{6}=15.5264.
\end{equation*}
Further computations show that
\begin{equation*}
\underset{0< c< 2,~ 0< r< 1}{\mathrm{max}} \; G(c,r,1) = \sqrt{\frac{15}{2}}+\frac{11}{6}=15.5264.
\end{equation*}

Now we find out the maximum values attained by $G(c,r,p)$ on the edges of $R$.
Evaluating (\ref{no19}) on the edge $c=0, p=1$ we obtain $G(0,r,1)=12(1- r^2)+4 r+3$. A simple computation shows that the maximum of $G(0,r,1)$ is $46/3$ which occurs at $r=1/6$. At the end points of this edge, we have $G(0,0,1)=15$ and $G(0,1,1)=7$. Hence
\begin{equation*}
\underset{0 \leq r \leq 1}{\mathrm{max}} G(0,r,1)= \frac{46}{3}.
\end{equation*}
In view of (\ref{no19}), we obtain by a series of straightforward computations the maximum value of $G(c,r,p)$ on the edges $c=0, r=0$; $c=0,r=1$ and $c=0, p=-1$ as
$$ \underset{-1 \leq p \leq 1}{\mathrm{max}} G(0,0,p)= 15, \qquad \underset{-1 \leq p \leq 1}{\mathrm{max}}  G(0,1,p)= 7 \qquad \mathrm{and} \qquad \underset{0 \leq r \leq 1}{\mathrm{max}} G(0,r,-1)= 15. $$
A simple observation shows that $G(2,r,p)=9$ implies
\begin{equation*}
\underset{-1 \leq p \leq 1}{\mathrm{max}} G(2,0,p)=\underset{-1 \leq p \leq 1}{\mathrm{max}} G(2,1,p)=\underset{0 \leq r \leq 1}{\mathrm{max}} G(2,r,-1)=\underset{0 \leq r \leq 1}{\mathrm{max}} G(2,r,1)=9.
\end{equation*}
As (\ref{no20}) is independent of $p$, the maximum value of $G(c,r,p)$ on the edges $r=0, p=-1$ and $r=0, p=1$ is
\begin{equation*}\label{newereq3}
\underset{0\leq c\leq 2}{\mathrm{max}} G(c,0,-1)=\underset{0\leq c\leq 2}{\mathrm{max}} G(c,0,1)=15.0858.
\end{equation*}
On the edge $r=1, p=-1$, (\ref{no21}) can be simplified to $G(c,1,-1)=|3 c^3+ c^2-9c-1|.$
A straightforward calculation shows that
\begin{equation*}
\underset{0\leq c\leq 2}{\mathrm{max}} G(c,1,-1)=9.
\end{equation*}
On the edge $r=1,p=1$,  (\ref{no21}) reduces to $G(c,1,1)=c^3-c^2-c+7$. A simple computation shows  that
\begin{equation*}
\underset{0\leq c\leq 2}{\mathrm{max}} G(c,1,1)=9.
\end{equation*}

Now we show that $G(c,r,p)$ does not attain  maximum value in the interior of the cuboid $R$.
In order to find the points where the maximum value is obtained in the interior of $R$, we solve $\frac{\partial G(c,r,p)}{\partial c}=0$, $\frac{\partial G(c,r,p)}{\partial r}=0$ and $\frac{\partial G(c,r,p)}{\partial p}=0$.
A computation shows that  $\frac{\partial G(c,r,p)}{\partial p}=0$ implies
\begin{equation}\label{newereq4}
p=\frac{3 c^4 r^2+c^4+3 c^3 r^2+c^3-12 c^2 r^2+2 c^2-12 c r^2+8 c+6}{6 c \left(c^3+2 c+6\right) r}.
\end{equation}
By substituting (\ref{newereq4}) in $\frac{\partial G(c,r,p)}{\partial r}=0$, we get
\begin{equation}\label{newereq5}
 r=\frac{\sqrt{c^3+2 c+6}}{\sqrt{3} \sqrt{c^3-4 c}}.
\end{equation}
It is easy to see that $\frac{c^3+2 c+6}{3 (c^3-4 c)}$ is negative for all values of $c \in [0,2]$. Hence there cannot be an extremum inside the cuboid $R$.
This shows that the maximum value of $|\gamma_3|$ is $\frac{1}{48} \left(5 \sqrt{\frac{15}{2}}+\frac{11}{6}\right)$ for $(c,r,p)=\left(\frac{1}{2} \left(6-\sqrt{30}\right),\frac{1}{105} \left(25-\sqrt{30}\right),1\right)$.

Let $c=c_1$ and $(c,r,p)=\left(\frac{1}{2} \left(6-\sqrt{30}\right),\frac{1}{105} \left(25-\sqrt{30}\right),1\right)$. Then in view of Lemma \ref{thm-4} we obtain $c_2=\frac{1}{12} \left(76-13 \sqrt{30}\right)$ and $c_3=\frac{1}{72} \left(554-75 \sqrt{30}\right)$.
It is not difficult to see that a function $G^*\in \mathcal{P}$ having
$$ (c_1,c_2,c_3)=\left(\frac{1}{2} \left(6-\sqrt{30}\right),\frac{1}{12} \left(76-13 \sqrt{30}\right) ,\frac{1}{72} \left(554-75 \sqrt{30}\right)\right) $$
is given by $G^*(z)=H_{t_1,\mu_1}(z)$ where
$\mu_1 = \frac{1}{12} \left(-1-\sqrt{30}\right)+i \frac{1}{12} \sqrt{113-2 \sqrt{30}} $,
and
$t_1=\frac{3}{278} \left(15 \sqrt{30}-56\right).$
Therefore the bound in ({\rm iii}) is sharp for the function $F_1(z)$ such that
$$zF'_1(z)=\frac{z}{1-z}G^*(z).
$$
\end{proof}
%

\begin{thm}\label{thm-6}
Let $f \in \mathcal{F}_2$ be given by (\ref{i01}). Then
\begin{enumerate}
\item[{(i)}] $|\gamma_1|\leq \frac{1}{4}$,
\item[{(ii)}] $|\gamma_2|\leq \frac{1}{2}$.
\item[{(iii)}] If $0\leq a_2\leq 1$ then $|\gamma_3| \leq \frac{1}{972} \left(95+23 \sqrt{46}\right)$.
\end{enumerate}
The inequalities are sharp.
\end{thm}

\begin{proof} Let $f \in \mathcal{F}_2$. It is evident that $f$ is close-to-convex with respect to the starlike function $g(z)=z/(1-z^2).$ From  (\ref{eq7}), $f(z)$ can be written as
\begin{equation}\label{neweq3}
zf'(z)=\frac{z}{1-z^2}h(z).
\end{equation}
Thus in view of  (\ref{neweq3}),  (\ref{gamma11}) reduces to
\begin{equation}\label{gamma1prob2}
4|\gamma_1| \leq |c_1|.
\end{equation}
Noting that $|c_1| \leq 2$, (\ref{gamma1prob2})  then implies that  $|\gamma_1| \leq 1/2$. It is easy to see that A function $p \in \mathcal{P}$ having $|c_1| =2$ is given by $p(z)=L_{1,\theta}(z)$ for $0 \leq \theta <2\pi$.
Substituting $L_{1,\theta}(z)$ in place of $h(z)$ in (\ref{neweq3}) shows that ${\rm (i)}$ is sharp.

A comparison of   (\ref{neweq3}) and (\ref{eq7}) shows that  (\ref{gamma2}) reduces to
\begin{equation*}\label{gamma2prob2}
6\gamma_2 \leq \left(c_2-\frac{3}{8}c_1^2\right)+1.
\end{equation*}
Applying the triangle inequality in conjunction with Lemma \ref{lem3} with $\mu =3/8$ shows that $|\gamma_2|\leq 1/2$. It is evident from Lemma \ref{lem3} that the equality holds in ${\rm (ii)}$  for the function $\widetilde{F_2}(z)$ such that
$z\widetilde{{F_2}'}(z)=z(1-z^2)^2L_{0,0}(z).
$

Considering (\ref{neweq3}) as an instance of (\ref{eq7}), (\ref{modgamma3}) can be simplified to
\begin{equation}\label{ed11}
96|\gamma_3| \leq 6\left(4-c^2\right)\left(1-r^2\right)+c\sqrt{\phi_2\left(c,r,p\right)},
\end{equation}
where
\begin{align*}
\phi_2(c,r,p) &=\left(c^2+4\right)^2+2r(4-c^2)(4+c^2)(2p+3r-6p^2r) \\
&\quad \quad +r^2(4-c^2)^2(4+9r^2-12rp).
\end{align*}
Let $F(c,r,p)=6(1-r^2)(4-c^2)+c\sqrt{\phi_2(c,r,p)}$. We find points where $F(c,r,p)$ attains the maximum value by finding its local maxima on the six faces and in the interior of $R$.
On the face $c=0$, $F(c,r,p)$ becomes
\begin{equation}\label{ed12}
F(0,r,p)=24 \left(1-r^2\right).
\end{equation}
As $F(0,r,p)$ is a decreasing function of $r$,  the maximum value of $F(0,r,p)$ is attained on the edge $c=0, r=0$. Consequently, we have
%
$$\underset{0\leq r\leq 1,~ -1\leq p\leq 1}{\mathrm{max}} \; F(0,r,p)=24.$$
On the face $c=2$, $F(c,r,p)$ becomes $F(2,r,p)=16$ and hence
$$ \underset{0\leq r\leq 1,~  -1\leq p \leq 1}{\mathrm{max}} \; F(2,r,p)=16. $$
On the face $r=0$, we can simplify $F(c,r,p)$ as
\begin{equation}\label{ed13}
F(c,0,p)=24-6 c^2+c\left(c^2+4\right).
\end{equation}
Since $F(c,0,p)$ is independent of $p$,  we find the set of all points where $\frac{\partial F(c,0,p)}{\partial c}$ vanishes as
$\{\frac{2}{3}\left(3-\sqrt{6}\right)\}\times \{0\}\times [-1,1]$ and hence $F\left(\frac{2}{3} \left(3-\sqrt{6}\right),0,p\right)=\frac{16}{9}\left(9+2\sqrt{6}\right)=24.7093$.
%
Evaluating  (\ref{ed13}) on the edges $c=0, r=0$ and $c=2, r=0$, we obtain
$$ \underset{0\leq c\leq 2,~ -1\leq p\leq 1}{\mathrm{max}} \; F(c,0,p) = 24.7093.$$
On the face $r=1$, $F(c,r,p)$ reduces to
\begin{equation}\label{ed14}
F(c,1,p)=2c\sqrt{24c^2(p-1)-16(p-1)(5+3p)+c^4(2-4p+3p^2)}.
\end{equation}
We solve $\frac{\partial F(c,1,p)}{\partial c}=0$ and $\frac{\partial F(c,1,p)}{\partial p}=0$ to determine points where maxima occur in the face $r=1$. A computation shows that $\frac{\partial F(c,1,p)}{\partial p}=0$ implies
\begin{equation}\label{pindFdp}
p=\frac{2 \left(c^2-2\right)}{3 \left(c^2+4\right)}.
\end{equation}
A slightly involved computation shows that $\frac{\partial F(c,1,p)}{\partial c}=0$ gives
\begin{equation}\label{dFdc}
(18c^4-96)p^2-8(3c^4-12c^2+8)p+(12c^4-96c^2+160)=0.
\end{equation}
Substituting (\ref{pindFdp}) in (\ref{dFdc}) followed by a computation gives
\begin{equation}\label{dFdc2}
\frac{4 \left(3 c^8-160 c^4-512 c^2+2048\right)}{3 \left(c^2+4\right)^2}=0.
\end{equation}
The numerical solution of (\ref{dFdc2}) in $0<c<2$ is $c\approx 1.54836$. Using (\ref{pindFdp}) we then obtain $p\approx 0.414152$.
Therefore $F(1.54836,1,0.414152)=18.0595$.

Using uniform continuity of $F(c,1,p)$ on $[0,2]\times \{1\}\times [-1,1]$ we infer that the difference between the maximum value  of $F(c,1,p)$ and $18.0595$ can be made smaller than an $\epsilon \ll 1$.
On the edge $c=0, r=1$, $F(c,r,p)$ becomes $F(0,1,p)=0$.
On the edge $c=2, r=1$, $F(c,r,p)$ becomes $F(2,1,p)=16$.
On the edge $r=1, p=-1$, (\ref{ed14}) can be simplified to $F(c,1,-1)=2c|3c^2-8|$. It is easy to see that $F(c,1,-1)$ has the maximum value $16$ on $[0,2]$.

A simple computation shows that the maximum value of $F(c,r,p)$ on the edge $r=1, p=1$ is 16.
Therefore,
$$ \underset{0\leq c\leq 2,~ -1\leq p\leq 1}{\mathrm{max}} \; F(c,1,p) \approx 18.0595.$$
On the face $p=-1$, $F(c,r,p)$ reduces to
\begin{align*}
F(c,r,-1)= 6(4-c^2)(1-r^2)+c|c^2+4-(2r-3r^2)(4-c^2)|.
\end{align*}
A computation similar to the one on  the  face $p=-1$ in Theorem \ref{thm-5} shows that $\frac{\partial F(c,r,-1)}{\partial c}=0$ and
$\frac{\partial F(c,r,-1) }{\partial r}=0$ have no solution in the interior of the face $p=-1$. Thus the maximum value is attained on the edges.

On the edge $c=0, p=-1$, $F(c,r,p)$ becomes $F(0,r,-1)=24(1-r^2)$. The maximum value  of $F(0,r,-1)$ is clearly $24$.
On the edge $r=0, p=-1$, $F(c,r,p)$ becomes
\begin{equation*}\label{newereq7}
F(c,0,-1)=6(4-c^2)+c(4+c^2).
\end{equation*}
The maximum value of $F(c,0,-1)$ is $\frac{16}{9}\left(9+2\sqrt{6}\right)=24.7093$ (see the face $r=0$).
The maximum values of $F(c,r,p)$ on the edges $c=2, p=-1$ and $r=1, p=-1$ are $16$ and $10.0566$ respectively (see the faces $c=2$ and $r=1$).
Therefore
\begin{equation*}
\underset{0\leq c\leq 2,~ 0\leq r\leq 1}{\mathrm{max}} \; F(c,r,-1) = \frac{16}{9}\left(9+2\sqrt{6}\right)=24.7093.
\end{equation*}
On the face $p=1$, $F(c,r,p)$ reduces to
\begin{equation*}
F(c,r,1)= 6(4-c^2)(1-r^2)+c|c^2+4+(2r+3r^2)(4-c^2)|.
\end{equation*}
Solving $\frac{\partial F(c,r,1)}{\partial c}=0$ and $\frac{\partial F(c,r,1) }{\partial r}=0$ we obtain $(c,r) = \left(\frac{1}{3} \left(8-\sqrt{46}\right),\frac{1}{75} \left(11-\sqrt{46}\right)\right)$ and hence $F\left(\frac{1}{3} \left(8-\sqrt{46}\right),\frac{1}{75} \left(11-\sqrt{46}\right),1\right)=\frac{8}{81}\left(95+23\sqrt{46}\right)=24.7895.$
It is not difficult to see that the maximum value of $F(c,r,1)$ on the  edges is $24.7093$, which occurs on the edge $r=0$, $p=1$ (see the face $r=0$)
as the computations for the  edges have been done on earlier faces. Therefore
\begin{equation*}
\underset{0\leq c\leq 2,~ 0\leq r\leq 1}{\mathrm{max}} \; F(c,r,1)=\frac{8}{81}\left(95+23\sqrt{46}\right)=24.7895.
\end{equation*}
We now show that $F(c,r,p)$ cannot attain a maximum in the interior of the cuboid $R$.
To determine points in the interior of $R$ where the maxima occurs (if any), we solve $\frac{\partial F(c,r,p)}{\partial c}=0$, $\frac{\partial F(c,r,p)}{\partial r}=0$ and $\frac{\partial F(c,r,p)}{\partial p}=0$.
A computation shows that $\frac{\partial F(c,r,p)}{\partial p}=0$ implies
\begin{equation}\label{newereq8}
p=\frac{3 c^2 r^2+c^2-12 r^2+4}{6 \left(c^2+4\right) r}.
\end{equation}
Using (\ref{newereq8}) in $\frac{\partial F(c,r,p)}{\partial r}=0$ and then solving for $r$ yields
$$ r=\frac{\sqrt{c^2+4}}{\sqrt{3} \sqrt{c^2-4}}.
$$
As $\frac{c^2+4}{3(c^2-4)}$ is negative for all values of $c\in[0,2]$, there cannot be an extremum in the interior of $R$.
This proves that the maximum value of $|\gamma_3|$ is $\frac{1}{972}\left(95+23\sqrt{46}\right)$ for $(c,r,p)=\left(\frac{1}{3} \left(8-\sqrt{46}\right),\frac{1}{75} \left(11-\sqrt{46}\right),1\right)$.

Let $c=c_1$ and $(c,r,p)=\left(\frac{1}{3} \left(8-\sqrt{46}\right),\frac{1}{75} \left(11-\sqrt{46}\right),1\right)$. Then in view of Lemma \ref{thm-4}, we obtain $c_2=\frac{1}{27} \left(134-19 \sqrt{46}\right)$ and $c_3=\frac{2}{243} \left(721-71 \sqrt{46}\right)$.
It is not difficult to see that a function $F^*\in \mathcal{P}$ having
\begin{equation*}
(c_1,c_2,c_3)= \left({\tiny \frac{1}{3}} \left(8-\sqrt{46}\right),{\tiny \frac{1}{27}} \left(134-19 \sqrt{46}\right), {\tiny \frac{2}{243}} \left(721-71 \sqrt{46}\right)\right)
\end{equation*}
is given by $F^*(z)=H_{t_2,\mu_2}(z)$ where
$\mu_2 = \frac{1}{18} \left(-1-\sqrt{46}\right)+i\frac{1}{18} \sqrt{277-2 \sqrt{46}} \quad\mbox{ and }~  t_2 = \frac{1}{10} \left(\sqrt{46}-4\right).$
This shows that  the bound in ${\rm (iii)}$ is sharp for the function $F_2(z)$ such that
$zF_2'(z)=z(1-z^2)^{-1}F^*(z).
$
\end{proof}

%
%
%
%

\begin{thm}\label{thm-7}
Let $f \in \mathcal{F}_3$ be given by (\ref{i01}). Then
\begin{enumerate}
\item[{(i)}] $|\gamma_1|\leq \frac{3}{4}$,
\item[{(ii)}] $|\gamma_2|\leq \frac{2}{5}$.
\item[{(iii)}] If $1/2\leq a_2\leq 3/2$ then  $|\gamma_3| \leq \frac{743+131\sqrt{262}}{7776}$.
\end{enumerate}
The inequalities are sharp.
\end{thm}

\begin{proof}
Let $f \in \mathcal{F}_3$. Then $f$ is close-to-convex with respect to the starlike function $g(z)=z/(1-z+z^2)$.
In view of (\ref{eq7}), $f(z)$ can be written as
\begin{equation}\label{neweq4}
zf'(z) = \frac{z}{1-z+z^2}h(z).
\end{equation}
Therefore  (\ref{gamma11}) reduces to
\begin{equation}\label{gamma1prob3}
4|\gamma_1| \leq 1+|c_1|.
\end{equation}
Thus from   (\ref{gamma1prob3}) we obtain  $|\gamma_1| \leq 3/4$ as $|c_1| \leq 2$ for $h\in \mathcal{P}$. A function in $\mathcal{P}$ having $|c_1|=2$ is given by $L_{1,\theta}(z)$, $0\leq \theta <2\pi$
The equality in ${\rm(i)}$ is attained for a function $\widetilde{f}(z)$ such that
$z \widetilde{f}'(z)=z(1-z+z^2)^{-1}L_{1,\theta}(z).$

In view of (\ref{neweq4}),  (\ref{modgamma2}) becomes
\begin{equation}\label{gamma2prob3}
6|\gamma_2| \leq 2-\frac{|c_1|^2}{2}+\frac{1}{8}\sqrt{(d^2+1-2dt)(d^2+9+6dt)}=:k(d,q).
\end{equation}
It is evident from (\ref{gamma2prob3}) that it is sufficient to find the maximum value of $k(d,q)$ in the square $S$ to obtain the same for $|\gamma_2|$.

To obtain points where $k(d,q)$ attains  maximum, we solve $\frac{\partial k(d,q)}{\partial d}=0$ and $\frac{\partial k(d,q)}{\partial q}=0$. The solutions obtained are complex, showing that $k(d,q)$ does not attain  maximum in the interior of $S$.

On the side $d=0$, $k(d,q)$ reduces to $k(d,q)=2.375$. On the side $d=2$, we see that $k(d,q)=(\sqrt{65+8t-48t})/8$. An elementary computation shows that $\underset{-1 \leq q\leq 1}{\mathrm{max}}\;k(2,q)=k(2,1/12)=1.01036.$

On the side $q=-1$, $k(d,q)$ becomes $k(d,-1)=(19+2d-5d^2)/8$. A straightforward computation shows that $\underset{0\leq d \leq 2}{\mathrm{max}}\;k(d,-1)=k(1/5,-1)=12/5=2.4$.

On the side $q=1$, $k(d,q)$ may be simplified as $k(d,1)=(19-2d-5d^2)/8$. As $k(d,1)$ is a decreasing function for $d\in [0,2]$, we see that $\underset{0\leq d \leq 2}{\mathrm{max}}\;k(d,1)=k(0,1)=19/8=2.375$.

Thus the maximum value of $k(d,q)$ in $S$ is $12/5$ and occurs at $(d,q)=(1/5,-1)$. Consequently, (\ref{gamma2prob3}) implies that $|\gamma_2|\leq 2/5$, with the equality occurring for $c_1=-1/5$.

Therefore, in view of Lemma \ref{lem2},  the equality in ${\rm (ii)}$ holds for the function  $\widetilde{F_2}(z)$ such that
$z\widetilde{{F_2}'}(z)=z(1-z+z^2)^{-1}L_{t,\theta}(z)$  where $t=1/10$ and $\theta =\pi$.

Using (\ref{neweq4}) we may rewrite (\ref{modgamma3}) as
\begin{equation}\label{ed16}
96 |\gamma_3| \leq 6(1-r^2)(4-c^2)+\sqrt{\phi_3(c,r,p)}
\end{equation}
where
\begin{align*}
\phi_3(c,r,p) &= (c^3-2c-10)^2+2r(4-c^2)(c^3-2c-10)(2p+2cp-6crp^2+3rc) \\
   &\quad \quad \quad +r^2(4-c^2)^2(4c^2+4+9c^2r^2+8c-12c^2rp-12crp).
\end{align*}
Let $K(c,r,p)=6(1-r^2)(4-c^2)+\sqrt{\phi_3(c,r,p)}$.
We find the points in the cuboid $R$ where the maxima of $K(c,r,p)$ occur.

On the face $c=0$, $K(c,r,p)$ takes the following form
\begin{equation}\label{ed17}
K(0,r,p)=24(1-r^2)+2\sqrt{25-40rp+16r^2}.
\end{equation}
By solving $\frac{\partial K(0,r,p)}{\partial r}=0$ and $\frac{\partial K(0,r,p)}{\partial p}=0$ we obtain $(r,p)=(0,0)$. Thus $K(c,r,p)$ does not attain  maximum in the interior of the face $c=0$.

On the face $c=2$, $K(c,r,p)$ reduces to $K(2,r,p)=6$ and hence
\begin{equation*}
\underset{0< r< 1,~ -1< p< 1}{\mathrm{max}} \; K(2,r,p) = 6.
\end{equation*}
On the face $r=0$, $K(c,r,p)$ may be simplified as
\begin{equation}\label{ed18}
K(c,0,p)=6(4-c^2)+|c^3-2c-10|.
\end{equation}
Since $K(c,0,p)$ is independent of $p$, it suffices to find out points such that $\frac{\partial K(c,0,p)}{\partial c}=0$. The set of all such points is $\{\frac{1}{3}\left(-6+\sqrt{42}\right)\}\times \{0\}\times [-1,1]$ and $$K\left(\frac{1}{3}\left(-6+\sqrt{42}\right),0,p\right)=\frac{14}{9}\left(9+2\sqrt{42}\right)=34.1623.
$$

Therefore
\begin{equation*}
\underset{0< c< 2,~ -1< p< 1}{\mathrm{max}} \; K(c,0,p) = {\tiny \frac{14}{9}}\left(9+2\sqrt{42}\right)=34.1623.
\end{equation*}

On the face $r=1$, $K(c,r,p)$ becomes
\begin{equation}\label{ed19}
K(c,1,p) = \sqrt{\psi_3(c,p)+2 \left(c^3-2 c-10\right) \left(c^2-4\right) \left(6cp^2-2cp-2p-3c\right)}
\end{equation}
where
\begin{equation*}
\psi_3(c,p) = \left(c^3-2 c-10\right)^2+\left(c^2-4\right)^2 (13c^2-12c^2p+8c-12cp+4).
\end{equation*}
A computation shows that $\frac{\partial K(c,1,p)}{\partial p}=0$ implies
\begin{equation}\label{pindKdp}
p=\frac{2 c^4+2 c^3-7 c^2-12 c-5}{3 c \left(c^3-2 c-10\right)}.
\end{equation}
A lengthy computation shows that $\frac{\partial K(c,1,p)}{\partial c}$ implies
\begin{alignat}{2}
&(9c^5-36c^3-45c^2+24c+60)p^2-(12c^5+10c^4-60c^3-60c^2+46c+48)p \label{dKdc}\\
&+(6 c^5+5 c^4-34 c^3-9 c^2+33 c-9)=0. \nonumber
\end{alignat}
Substituting (\ref{pindKdp}) in (\ref{dKdc}) and then performing another  lengthy computation gives
\begin{equation}\label{dKdc2}
 \frac{(c^3-5c+5)\zeta_2 (c)}{3 c^2 \left(c^3-2 c-10\right)^2}=0
\end{equation}
where
\begin{equation*}
\zeta_2 (c)=6 c^{10}-5 c^9-32 c^8-104 c^7+147 c^6+375 c^5+459 c^4-375 c^3-1135 c^2+140 c+100.
\end{equation*}
The numerical solutions of (\ref{dKdc2}) are obtained as $c \approx 0.354278$ and $c \approx 1.27688$.
%
%
Further computations show that $K(c,1,p)$ does not attain a maxima at these points even though the partial derivatives vanish.
On the face $p=-1$, $K(c,r,p)$ reduces to
\begin{equation*}\label{ed20}
K(c,r,-1)=6(1-r^2)(4-c^2)-(c^3-2c-10)+2(4-c^2)(2+2c+3cr^2).
\end{equation*}
By solving $\frac{\partial K(c,r,-1)}{\partial c}=0$ and $\frac{\partial K(c,r,-1)}{\partial r}=0$ we obtain $c=\frac{1}{6}(-14+\sqrt{262})$ and $r=\frac{1}{69}(3+\sqrt{262})$. The corresponding maximum value is $$K\left(\frac{1}{6}(-14+\sqrt{262}),\frac{1}{69}(3+\sqrt{262}),-1\right)=\frac{1}{81}(743+131\sqrt{262})=35.3509.
$$
Therefore
$$ \underset{0< c< 2, ~ 0< r< 1}{\mathrm{max}} \; K(c,r,-1)= 35.3509.$$

On the face $p=1$, $K(c,r,p)$ reduces to
\begin{equation}
K(c,r,1)=6(1-r^2)(4-c^2)-(c^3-2c-10)+(4-c^2)(3cr^2-2-2c).
\end{equation}
It is not difficult to see that $\frac{\partial K(c,r,1)}{\partial c}=0$ and $\frac{\partial K(c,r,1)}{\partial r}=0$ have no solution in the interior of the face $p=1$. Thus $K(c,r,p)$ does not attain maximum in the interior of this face.

Now we find the maximum values attained on the edges of R.
It is evident from (\ref{ed17}) that on the edges $c=0, r=0$ and $c=0, r=1$, the maximum values of $K(c,r,p)$ are \begin{equation*}
\underset{-1\leq p\leq 1}{\mathrm{max}} \;K(0,0,p)=34 \qquad \mathrm{and} \qquad \underset{-1\leq p\leq 1}{\mathrm{max}} \; K(0,1,-1)=18.
\end{equation*}
On the edge $c=0, p=-1$, (\ref{ed17}) reduces to $K(0,r,-1)=24(1-r^2)+2(5+4r)$. An elementary computation shows that the maximum value of $K(0,r,-1)$ is attained at $\left(0,{\tiny \frac{1}{6}},-1\right)$ and $\underset{0\leq r\leq 1}{\mathrm{max}} \; K(0,r,-1)=104/3$.

On the edge $c=0, p=1$, (\ref{ed17}) reduces to $K(0,r,1)=24 \left(1-r^2\right)+2 (5-4 r)$. A computation shows that $\underset{0\leq r\leq 1}{\mathrm{max}} \; K(0,r,1)=34.$

It is evident that $K(2,r,p)=6$ implies
\begin{equation*}
\underset{-1 \leq p \leq 1}{\mathrm{max}} K(2,0,p)=\underset{-1 \leq p \leq 1}{\mathrm{max}} K(2,1,p)=\underset{0 \leq r \leq 1}{\mathrm{max}} K(2,r,-1)=\underset{0 \leq r \leq 1}{\mathrm{max}} K(2,r,1)=6.
\end{equation*}

Considering (\ref{ed18}) and the maximum value of $K(c,0,p)$ we obtain the maximum values on the edges $r=0, p=-1$ and $r=0,p=1$ as
\begin{equation*}
\underset{0\leq c\leq 2}{\mathrm{max}} \; K(c,0,-1)=\underset{0\leq c\leq 2}{\mathrm{max}} \; K(c,0,1)=34.1623.
\end{equation*}

On the edge $r=1,p=-1$, (\ref{ed19}) maybe be simplified as
\begin{alignat*}{1}
& \qquad K(c,1,-1) \\
&=\sqrt{\left(c^3-2 c-10\right)^2-2(5c+2) \left(c^3-2 c-10\right) \left(4-c^2\right) +\left(25 c^2+20 c+4\right) \left(4-c^2\right)^2}.
\end{alignat*}
A computation shows that $K(c,1,-1)$ attains the local maximum at $(1,1,-1)$ and $\underset{0\leq c\leq 2}{\mathrm{max}} \; K(c,1,-1)=32$.

On the edge $r=1, p=1$, (\ref{ed19}) reduces to $K(c,1,1)=2(1+3c+c^2-c^3)$.
An elementary computation shows that
\begin{equation*}
\underset{0\leq c\leq 2}{\mathrm{max}} \; K(c,1,1)=K\left({\tiny \frac{1}{3}}(1+\sqrt{10}),1,1\right)={\tiny \frac{8}{27}}(14+5\sqrt{10})=8.833.
\end{equation*}

Now we show that $K(c,r,p)$ does not attain  maximum in the interior of the cuboid $R$.
At the points where the maxima occur in the cuboid $R$ we have $\frac{\partial K(c,r,p)}{\partial c}=0, \frac{\partial K(c,r,p)}{\partial r}=0$ and $\frac{\partial K(c,r,p)}{\partial p}=0$. A computation shows that
$\frac{\partial K(c,r,p)}{\partial p}=0$ implies
\begin{equation}\label{ed21}
p=\frac{3 c^4 r^2+c^4+3 c^3 r^2+c^3-12 c^2 r^2-2 c^2-12 c r^2-12 c-10}{6 c \left(c^3-2 c-10\right) r}.
\end{equation}
Substituting (\ref{ed21}) in  $\frac{\partial K(c,r,p)}{\partial r}=0$ and then solving for $r$ we obtain
\begin{equation}\label{ed22}
r=\frac{\sqrt{c^3-2 c-10}}{\sqrt{3 c^3-12 c}}.
\end{equation}
Substituting (\ref{ed22}) in (\ref{ed21}) gives
\begin{equation}\label{ed23}
p=\frac{(c+1) \sqrt{c \left(c^2-4\right)}}{\sqrt{3} c \sqrt{c^3-2 c-10}}.
\end{equation}
Substituting  (\ref{ed22}) and (\ref{ed23}) in  $\frac{\partial K(c,r,p)}{\partial c}=0$, we obtain
$$ \frac{8 \left(c^3-5 c+5\right)}{c^2-4} = 0. $$
It can be seen that the roots to the above equation are either negative or imaginary.
This shows that a maximum cannot be attained inside $R$.
Thus we see that the maximum value for $|\gamma_3|$ is attained for
$$(c,r,p)=\left(\frac{1}{6}(-14+\sqrt{262}),\frac{1}{69}(3+\sqrt{262}),-1\right)
$$
and is equal to $(743+131\sqrt{262})/81=35.3509$.
Using these values of $(c,r,p)$ in Lemma \ref{thm-4}, we obtain $c_2=\frac{1}{108} \left(548-37 \sqrt{262}\right)$ and
$c_3=\frac{47525 \sqrt{262}-698926}{44712}$. Therefore for given
$$(c_1,c_2,c_3)=\left(\frac{1}{6}(-14+\sqrt{262}),\frac{1}{108} \left(548-37 \sqrt{262}\right),\frac{47525 \sqrt{262}-698926}{44712}\right)
$$
there exists a function $K^* \in \mathcal{P}$ given by $K^*(z)=H_{t_3,\mu_3}(z)$, where
$$\mu_3 = \frac{-769+35\sqrt{262}}{828}+i\frac{\sqrt{-226727+53830\sqrt{262}}}{828} \quad\mbox{ and } t_3=\frac{32352-687\sqrt{262}}{64622}.
$$
The inequality ({\rm iii}) is sharp for the function $F_3(z)$ such that
$$zF_3'(z)=\frac{z}{1-z+z^2}K^*(z).
$$
\end{proof}

\end{document}